   \pdfoutput=1 

\documentclass{article}
\usepackage{excludeonly}

\usepackage[disable]{todonotes}

\usepackage{savesym}
\usepackage{scrextend}  
\usepackage{amsthm,amsmath}
\usepackage{amssymb,latexsym,graphicx}
\usepackage{amscd}
 \usepackage[all,cmtip]{xy}
\usepackage{tikz-cd}
  \usetikzlibrary{arrows}
\tikzset{
  commutative diagrams/.cd,
  arrow style=tikz
  }
\usepackage{tikz}

\usepackage{accents}
\usepackage{cite}
\usepackage{mathtools}
\usepackage{stmaryrd} 

\usepackage{calligra}
\DeclareMathAlphabet{\mathcalligra}{T1}{calligra}{m}{n}
\DeclareMathAlphabet{\mathpzc}{OT1}{pzc}{m}{it}
\usepackage{hycolor}
\usepackage{xcolor}
\usepackage[
                       colorlinks=true,
                       linkcolor=black, 
                       citecolor=black, 
                       urlcolor=blue,
                     ]{hyperref}
\usepackage{soul} 

\usepackage{makeidx}

\usepackage[intoc]{nomentbl}
\makenomenclature

\usepackage{stackengine} 
\usepackage{booktabs} 

\newtheorem{theoremABC}{Theorem}

\newtheorem{theorem}{Theorem}[section]

\newtheorem{lemma}[theorem]{Lemma}
\newtheorem{proposition}[theorem]{Proposition}

\theoremstyle{definition}

\newtheorem{remark}[theorem]{Remark}

\theoremstyle{remark}

\newcommand{\C}{{\mathbb{C}}}

\newcommand{\N}{{\mathbb{N}}}

\newcommand{\R}{{\mathbb{R}}}
\renewcommand{\SS}{{\mathbb{S}}}

\newcommand{\Z}{{\mathbb{Z}}}
\newcommand{\Dd}{{\mathcal{D}}}

\newcommand{\Ii}{{\mathcal{I}}}

\newcommand{\Ll}{{\mathcal{L}}}   
\newcommand{\Mm}{{\mathcal{M}}}   
\newcommand{\Nn}{{\mathcal{N}}}

\newcommand{\Rr}{{\mathcal{R}}}

\newcommand{\id}{{\rm id}}                

\newcommand{\cgraph}[1]{\Gamma_{\kern-.5ex{}#1}}     
\newcommand{\Diff}{{\rm Diff}}        

\newcommand{\norm}{{\rm norm}}

\newcommand{\SC}{{\mathrm{sc}}}                  
\newcommand{\SSC}{{\mathrm{ssc}}}               
\def\Nablatop#1{\nabla^{#1}\kern-.5ex{}}

\def\NABLA#1{{\mathop{\nabla\kern-.5ex\lower1ex\hbox{$#1$}}}}
\def\Nabla#1{\nabla\kern-.5ex{}_{#1}}
\def\Tabla#1{\Tilde\nabla\kern-.5ex{}_{#1}}
\def\Babla#1{\widebar\nabla\kern-.5ex{}_{#1}}
\def\abs#1{\mathopen|#1\mathclose|}   
\def\Abs#1{\left|#1\right|}            
\def\norm#1{\mathopen\|#1\mathclose\|}

\renewcommand{\Tilde}{\widetilde}

\newlength\eqshift
\renewcommand\theequation{\thesection.\arabic{equation}}
\let\savetheequation\theequation

\makeatletter
\renewcommand*\env@matrix[1][\arraystretch]{%
  \edef\arraystretch{#1}%
  \hskip -\arraycolsep
  \let\@ifnextchar\new@ifnextchar
  \array{*\c@MaxMatrixCols c}}
\makeatother

\makeatletter
\let\save@mathaccent\mathaccent
\newcommand*\if@single[3]{%
  \setbox0\hbox{${\mathaccent"0362{#1}}^H$}%
  \setbox2\hbox{${\mathaccent"0362{\kern0pt#1}}^H$}%
  \ifdim\ht0=\ht2 #3\else #2\fi
  }
\newcommand*\rel@kern[1]{\kern#1\dimexpr\macc@kerna}
\newcommand*\widebar[1]{\@ifnextchar^{{\wide@bar{#1}{0}}}{\wide@bar{#1}{1}}}
\newcommand*\wide@bar[2]{\if@single{#1}{\wide@bar@{#1}{#2}{1}}{\wide@bar@{#1}{#2}{2}}}
\newcommand*\wide@bar@[3]{%
  \begingroup
  \def\mathaccent##1##2{%
    \let\mathaccent\save@mathaccent
    \if#32 \let\macc@nucleus\first@char \fi
    \setbox\z@\hbox{$\macc@style{\macc@nucleus}_{}$}%
    \setbox\tw@\hbox{$\macc@style{\macc@nucleus}{}_{}$}%
    \dimen@\wd\tw@
    \advance\dimen@-\wd\z@
    \divide\dimen@ 3
    \@tempdima\wd\tw@
    \advance\@tempdima-\scriptspace
    \divide\@tempdima 10
    \advance\dimen@-\@tempdima
    \ifdim\dimen@>\z@ \dimen@0pt\fi
    \rel@kern{0.6}\kern-\dimen@
    \if#31
      \overline{\rel@kern{-0.6}\kern\dimen@\macc@nucleus\rel@kern{0.4}\kern\dimen@}%
      \advance\dimen@0.4\dimexpr\macc@kerna
      \let\final@kern#2%
      \ifdim\dimen@<\z@ \let\final@kern1\fi
      \if\final@kern1 \kern-\dimen@\fi
    \else
      \overline{\rel@kern{-0.6}\kern\dimen@#1}%
    \fi
  }%
  \macc@depth\@ne
  \let\math@bgroup\@empty \let\math@egroup\macc@set@skewchar
  \mathsurround\z@ \frozen@everymath{\mathgroup\macc@group\relax}%
  \macc@set@skewchar\relax
  \let\mathaccentV\macc@nested@a
  \if#31
    \macc@nested@a\relax111{#1}%
  \else
    \def\gobble@till@marker##1\endmarker{}%
    \futurelet\first@char\gobble@till@marker#1\endmarker
    \ifcat\noexpand\first@char A\else
      \def\first@char{}%
    \fi
    \macc@nested@a\relax111{\first@char}%
  \fi
  \endgroup
}
\makeatother

\def\XXint#1#2#3{{\setbox0=\hbox{$#1{#2#3}{\int}$}
     \vcenter{\hbox{$#2#3$}}\kern-.5\wd0}}

\long\def\symbolfootnote[#1]#2{\begingroup%
\def\thefootnote{\fnsymbol{footnote}}\footnote[#1]{#2}\endgroup}

\tikzset{
  symbol/.style={
    draw=none,
    every to/.append style={
      edge node={node [sloped, allow upside down, auto=false]{$#1$}}}
  }
}

\begin{document}
\sloppy

\author{\quad Urs Frauenfelder \quad \qquad\qquad
             Joa Weber\footnote{
  Email: urs.frauenfelder@math.uni-augsburg.de
  \hfill
  joa@unicamp.br
  }
        %
        %
    \\
    Universit\"at Augsburg \qquad\qquad
    UNICAMP
}

\title{Loop space blow-up and scale calculus}

\date{\today}

\maketitle 
%


%
%

%





\begin{abstract}
In this note we show that the
Barutello-Ortega-Verzini regularization map
is scale smooth.
\end{abstract}


\subsection*{Introduction}

Regularization of two-body collisions
is an important topic in celestial mechanics
and the dynamics of electrons in atoms.
Most classical regularizations blow up the energy
hypersurfaces to regularize collisions.
Recently Barutello, Ortega, and Verzini~\cite{Barutello:2021b}
discovered a new regularization technique
which does not blow up the energy hypersurface,
but instead of that the loop space.
The discovery is explained in detail
in~\cite[\S 2.2]{Frauenfelder:2021e}
in case of the free fall.

This new regularization technique
is in particular useful for non-autonomous systems
for which there is no preserved energy and therefore
no energy hypersurface which can be blown up.

\subsection*{Setup and main result}

Let $\C^\times:=\C\setminus\{0\}$ be the punctured complex plane.
We abbreviate by
$$
   \Ll\C^\times:=C^\infty(\SS^1,\C^\times)
   ,\qquad
   \SS^1:=\R/\Z ,
$$
the space of smooth loops in the punctured plane.
Barutello-Ortega-Verzini regularization is carried out
using a map $\Rr\colon\Ll\C^\times\to\Ll\C^\times$
defined as follows.
For $z\in\Ll\C^\times$ define the map
$$
   t_z\colon\SS^1\to\SS^1
   ,\quad
   \tau\mapsto
   \tfrac{1}{\norm{z}_{L^2}^2}\int_0^\tau\Abs{z(s)}^2 ds\in[0,1] .
$$
Indeed it takes values in $\R/\Z$ since $t_z(0)=0$ and $t_z(1)=1$
coincide modulo $1$ which, by continuity of $t_z$, implies surjectivity.
The derivative
$$
   t_z^\prime(\tau)
   :=\tfrac{d}{d\tau} t_z(\tau)
   =\tfrac{1}{\norm{z}_{L^2}^2} \Abs{z(\tau)}^2
   >0
$$
depends continuously on $\tau$ and is anywhere strictly positive.
So $t_z$ is also injective. Therefore the inverse $\tau_z:={t_z}^{-1}$
exists and it is $C^1$; see~\cite[\S 2.2]{Frauenfelder:2021e}.
This shows that $t_z$ and $\tau_z$ are elements
of the circle's diffeomorphism group $\Diff(\SS^1)$.
Using this notion the \emph{Barutello-Ortega-Verzini map} or,
alternatively, the \textbf{rescale-square operation} is defined by
$$
   \Rr\colon \Ll\C^\times\to\Ll\C^\times
   ,\quad
   z\mapsto z^2\circ \tau_z
   =[t\mapsto z^2(\tau_z(t))] .
$$
The loop space $\Ll\C^\times$ is an open subset of the Fr\'{e}chet space
$\Ll\C=C^\infty(\SS^1,\C)$.
This Fr\'{e}chet space arises as the smooth level of the scale Hilbert space
$$
   \Lambda\C=(\Lambda\C_k)_{k\in\N_0}
   ,\qquad
   \Lambda\C_k:=W^{2+k,2}(\SS^1,\C),\quad k\in\N_0 .
$$
Indeed $\Lambda\C=\cap_{k\in\N_0}\Lambda\C_k$.
For details of scale Hilbert spaces and scale smoothness
($\SC^\infty$) see~\cite{Hofer:2021a};
for an introduction see~\cite{Weber:2019a},
for a summary of what we need here see~\cite{Frauenfelder:2021b}.

\smallskip
\noindent
The scale Hilbert space $\Lambda\C=W^{2,2}(\SS^1,\C)$
contains the open subset
$$
   \Lambda\C^\times:=W^{2,2}(\SS^1,\C^\times)
$$
which inherits the levels
$$
   \Lambda\C^\times_k
   :=\Lambda\C^\times\cap \Lambda\C_k
   =W^{2+k,2}(\SS^1,\C^\times)
   ,\quad k\in\N_0.
$$
Observe that $\Ll\C^\times=C^\infty(\SS^1,\C^\times)
=\cap_{k\in\N_0}\Lambda\C^\times_k$ is the smooth level of
$\Lambda\C^\times$.
The map $t_z$, hence $\Rr$, is well defined for
$z\in\Lambda\C^\times$.
In this note we prove

\begin{theoremABC}\label{thm:A}
The map $\Rr\colon\Lambda\C^\times\to\Lambda\C^\times$ is scale smooth.
\end{theoremABC}

\subsection*{Scale smoothness}

\subsubsection*{Neumeisters theorem}

The proof of Theorem~\ref{thm:A} is based on a result of
Neumeister which tells that the action on the free loop space of
the diffeomorphism group of the circle
$$
   \Dd
   :=\{\psi\in W^{2,2}(\SS^1,\SS^1)\mid
   \text{$\psi$ is bijective and $\psi^{-1}\in W^{2,2}(\SS^1,\SS^1)$}
   \}
$$%
with levels $\Dd_k:=\Dd\cap W^{2,2+k}(\SS^1,\SS^1)$,
for $k\in\N_0$, is scale smooth.

\begin{theorem}[\hspace{-.1pt}{\cite[Prop.\,3.2]{Neumeister:2021a}}]
\label{thm:Neumeister}
The reparametrization map
$$
   \rho\colon \Dd\times \Lambda\C\to\Lambda\C
   ,\quad
   (\psi,z)\mapsto z\circ \psi
$$
is scale smooth.
\end{theorem}

\begin{remark}[Why the zero level is chosen $W^{2,2}$ and not $W^{1,2}$]
In case $(\psi,z)\in W^{1,2}(\SS^1,\SS^1)\times W^{1,2}(\SS^1,\C)$,
the derivative
$$
    (z\circ \psi)^\prime
   =\underbrace{z^\prime|_\psi}_{L^2}\cdot \underbrace{\psi^\prime}_{L^2}
$$
is not necessarily in $L^2$. But in case $(\psi,z)\in W^{2,2}\times W^{2,2}$
the derivative lies in $W^{1,2}$ since both factor do and on one of them
we can use that $W^{1,2}\subset C^0$.
Then the second weak derivative exists as well
$$
    (z\circ \psi)^{\prime\prime}
   =\underbrace{z^{\prime\prime}|_\psi}_{L^2}
   \cdot\underbrace{\psi^\prime}_{W^{1,2}\subset C^0}
   \cdot\underbrace{\psi^\prime}_{W^{1,2}\subset C^0}
   +\underbrace{z^\prime|_\psi}_{W^{1,2}\subset C^0}
   \cdot \underbrace{\psi ^{\prime\prime}}_{L^2}
$$
and lies in $L^2$ as desired.
\end{remark}

\subsubsection*{Time rescaling}

\begin{lemma}\label{le:t_z}
The map
$$
   t\colon\Lambda\C^\times\to\Dd,\quad
   z\mapsto t_z
$$
is scale smooth.
\end{lemma}

\begin{proof}
We show that the map $t$ is \textbf{strongly scale smooth}
($\SSC^\infty$).
By definition, this means that the map $t$ is on each level $k\in\N_0$
smooth as a map $\Lambda\C_k^\times\to\Dd_k$.
But strongly scale smooth implies scale smooth ($\SC^\infty$).
\\
To this end we decompose the map $t$ as the
composition $t=\Mm\circ(\Ii,\iota\circ\Nn)$
of several maps each of which is obviously smooth.
These maps are
$$
   \Nn\colon \Lambda\C_k^\times\to(0,\infty)
   ,\quad
   z\mapsto\norm{z}_{L^2}^2
   \qquad
   \iota\colon(0,\infty)\to(0,\infty)
   ,\quad
   x\mapsto \tfrac{1}{x}
$$
and
$$
   \Ii\colon \Lambda\C_k^\times\to W^{k,2}([0,1],\R)
   ,\quad
   z\mapsto
   \left[\tau\mapsto \int_0^\tau\abs{z(s)}^2\, ds\right]
$$
and
$$
   \Mm\colon W^{k,2}([0,1],\R)\times\R\to W^{k,2}([0,1],\R)
   ,\quad
   (v,r)\mapsto rv .
$$
This proves Lemma~\ref{le:t_z}.
\end{proof}

\subsubsection*{Inverse}

\begin{proposition}\label{prop:inverse}
The map $I\colon \Dd\to\Dd$, $\psi\mapsto \psi^{-1}$, is scale smooth.
\end{proposition}

\begin{proof}
We compute the scale differentials of the inversion map $I$.
By definition of the inverse, for every $\psi\in\Dd$ we have the
identity $\psi\circ I(\psi)=\id$.
Differentiating this identity we obtain for a tangent vector
$\hat\psi\in T_\psi\Dd=W^{2,2}(\SS^1,\R)$ that
$$
   0=\hat\psi\circ I(\psi)+d\psi|_{I(\psi)} DI|_\psi \hat \psi
   =\hat\psi\circ\psi^{-1}+(\psi^\prime\circ\psi^{-1})\cdot DI|_\psi\hat \psi.
$$
Therefore we obtain the formula
$$
   DI|_\psi\hat \psi
   =\frac{-1}{\psi^\prime\circ \psi^{-1}} \cdot \hat\psi\circ\psi^{-1}.
$$
Note that $\psi^\prime(t)\not=0$, for every $t$,
since $\psi\colon\SS^1\to\SS^1$ is a diffeomorphism.

Let $\psi\in\Dd$ and $\hat\psi_1,\hat\psi_2\in T_\psi\Dd$.
Note that $\psi$ appears three times in the formula for $DI|_\psi\hat \psi$.
Hence the second derivative is a sum of three terms, namely
\begin{equation*}
\begin{split}
   D^2I|_\psi(\hat \psi_1,\hat\psi_2)
   &=\frac{1}{(\psi^\prime\circ \psi^{-1})^2}
   (\hat\psi^\prime_2\circ\psi^{-1})(\hat\psi_1\circ\psi^{-1})
   \\
   &\quad
   -\frac{1}{(\psi^\prime\circ \psi^{-1})^3} (\psi^{\prime\prime}\circ\psi^{-1})
   (\hat\psi_2\circ\psi^{-1}) (\hat\psi_1\circ\psi^{-1})
   \\
   &\quad
   +\frac{1}{(\psi^\prime\circ \psi^{-1})^2}
   (\hat\psi^\prime_1\circ\psi^{-1})(\hat\psi_2\circ\psi^{-1}) .
\end{split}
\end{equation*}
Note that $D^2I|_\psi(\hat \psi_1,\hat\psi_2)$ is a polynomial in
the six variables
$$
   \frac{1}{\psi^\prime\circ\psi^{-1}},\quad
   \psi^{\prime\prime}\circ\psi^{-1},\quad
   \hat\psi_1\circ\psi^{-1},\quad
   \hat\psi_2\circ\psi^{-1},\quad
   \hat\psi_1^{\prime}\circ\psi^{-1},\quad
   \hat\psi_2^{\prime\prime}\circ\psi^{-1}.
$$
If $\psi$ is in $W^{k+4,2}$ and $\hat\psi_1,\hat\psi_2$ are in $W^{k+3,2}$,
then all these variables are in $W^{k+2,2}$.
Since multiplication $W^{k+2,2}\times W^{k+2,2}\to W^{k+2,2}$ is continuous,
we conclude that the map
$$
   \Dd^{k+2}\times W^{k+3,2}\times W^{k+3,2}\to W^{k+2,2}
   ,\quad
   (\psi,\hat \psi_1,\hat\psi_2)\mapsto D^2I|_\psi(\hat \psi_1,\hat\psi_2)
$$
is continuous.
Therefore, by the criterium in~\cite[Le.\,4.8]{Frauenfelder:2021b}
the inversion map $I$ is of class $\SC^2$.

Differentiating further by induction we obtain that for every $n\in\N$
$D^nI|_\psi(\hat\psi_1,\dots,\hat\psi_n)$
is a polynomial in the $(n+1)n$ variables
\begin{equation*}
\begin{split}
   &\frac{1}{\psi^\prime\circ\psi^{-1}},\quad
   \psi^{\prime\prime}\circ\psi^{-1},\dots, \psi^{(n)}\circ\psi^{-1} ,
\\
   &\hat\psi_1\circ\psi^{-1},\dots, \hat\psi_1^{(n-1)}\circ\psi^{-1}
   ,\quad\dots,\quad
   \hat\psi_n\circ\psi^{-1},\dots, \hat\psi_n^{(n-1)}\circ\psi^{-1} .
\end{split}
\end{equation*}
Hence the map
\begin{equation*}
\begin{split}
   \Dd^{k+m}\times
   W^{k+m+1,2}\times \dots\times W^{k+m+1,2}&\to W^{k+2,2}
\\
   (\psi,\hat \psi_1,\dots,\hat\psi_m)
   &\mapsto D^mI|_\psi(\hat \psi_1,\dots,\hat\psi_m)
\end{split}
\end{equation*}
is continuous.
Therefore the map $I$ is $\SC^n$ for every $n\in\N$,
hence $\SC^\infty$.
This finishes the proof of Proposition~\ref{prop:inverse}.
\end{proof}

\begin{remark}
Together with Neumeisters Theorem~\ref{thm:Neumeister},
Proposition~\ref{prop:inverse} shows that the diffeomorphism group of
the circle is a scale Lie group.
\end{remark}

\subsubsection*{Proof of main result}

\begin{proof}[Proof of Theorem~\ref{thm:A}]
The map $\Rr$ can be written as the composition
$\Rr(z)=\rho(\sigma(z),I\circ t(z))$
of scale smooth maps and is therefore itself scale smooth
by the scale calculus chain rule~\cite[Thm.\,1.3.1]{Hofer:2021a}.

We abbreviate by $\sigma\colon \Lambda\C^\times\to\Lambda\C^\times$
the squaring map $z\mapsto z^2$.
The squaring map is obviously smooth on every level, hence
$\SSC^\infty$, thus $\SC^\infty$.
The map $I$ is $\SC^\infty$ by Proposition~\ref{prop:inverse}.
The map $t$ is $\SC^\infty$ by Lemma~\ref{le:t_z}.
The map $\rho$ is $\SC^\infty$ by Theorem~\ref{thm:Neumeister}.
This concludes the proof of Theorem~\ref{thm:A}.
\end{proof}

\bibliographystyle{alpha}
\addcontentsline{toc}{section}{References}
\bibliography{$HOME/Dropbox/0-Libraries+app-data/Bibdesk-BibFiles/library_math,$HOME/Dropbox/0-Libraries+app-data/Bibdesk-BibFiles/library_math_2020,$HOME/Dropbox/0-Libraries+app-data/Bibdesk-BibFiles/library_physics}{}

\begin{thebibliography}{{Neu}21}

\bibitem[BOV21]{Barutello:2021b}
Vivina Barutello, Rafael Ortega, and Gianmaria Verzini.
\newblock Regularized variational principles for the perturbed {K}epler
  problem.
\newblock {\em Adv. Math.}, 383:Paper No. 107694, 64, 2021.
\newblock \href{https://arxiv.org/abs/2003.09383}{arXiv:2003.09383}.

\bibitem[FW21a]{Frauenfelder:2021e}
Urs {Frauenfelder} and Joa {Weber}.
\newblock {The regularized free fall I -- Index computations}.
\newblock {\em Russian Journal of Mathematical Physics}, 28(4):464--487, 2021.
\newblock \href{https://rdcu.be/cCJqj}{SharedIt}.

\bibitem[FW21b]{Frauenfelder:2021b}
Urs {Frauenfelder} and Joa Weber.
\newblock {The shift map on Floer trajectory spaces}.
\newblock {\em J. Symplectic Geom.}, 19(2):351--397, 2021.
\newblock \href{https://arxiv.org/abs/1803.03826}{arXiv:1803.03826}.

\bibitem[HWZ21]{Hofer:2021a}
Helmut Hofer, Krzysztof Wysocki, and Eduard Zehnder.
\newblock {\em Polyfold and {F}redholm theory}, volume~72 of {\em Ergebnisse
  der Mathematik und ihrer Grenzgebiete. 3. Folge. A Series of Modern Surveys
  in Mathematics}.
\newblock Springer, Cham, 2021.
\newblock \href{https://arxiv.org/abs/1707.08941}{Preliminary version on
  arXiv:1707.08941}.

\bibitem[{Neu}21]{Neumeister:2021a}
Oliver {Neumeister}.
\newblock {The curve shrinking flow, compactness and its relation to scale
  manifolds}.
\newblock {\em arXiv e-prints}, 2021.
\newblock \href{https://arxiv.org/abs/2104.12906}{arXiv:2104.12906}.

\bibitem[Web19]{Weber:2019a}
Joa Weber.
\newblock {\em Scale Calculus and M-Polyfolds -- An Introduction}.
\newblock Publica\c{c}\~oes Matem\'aticas do IMPA. [IMPA Mathematical
  Publications]. Instituto Nacional de Matem\'atica Pura e Aplicada (IMPA), Rio
  de Janeiro, 2019.
\newblock 32${^{\rm{o}}}$ Col\'oquio Brasileiro de Matem\'atica.
  \href{https://impa.br/wp-content/uploads/2022/03/32CBM03_eBook.pdf}{Access
  pdf}. Extended version in preparation.

\end{thebibliography}

%


\end{document}